\documentclass[10pt]{article}
\font\smallit=cmti10
\font\smalltt=cmtt10

\usepackage{amssymb,latexsym,amsmath,epsfig,amsthm} 
\usepackage{tikz}

\makeatletter

\renewcommand\section{\@startsection {section}{1}{\z@}
{-30pt \@plus -1ex \@minus -.2ex}
{2.3ex \@plus.2ex}
{\normalfont\normalsize\bfseries\boldmath}}

\renewcommand\subsection{\@startsection{subsection}{2}{\z@}
{-3.25ex\@plus -1ex \@minus -.2ex}
{1.5ex \@plus .2ex}
{\normalfont\normalsize\bfseries\boldmath}}

\renewcommand{\@seccntformat}[1]{\csname the#1\endcsname. }

\makeatother

\newtheorem{theorem}{Theorem}
\newtheorem{lemma}{Lemma}

\newtheorem{corollary}{Corollary}
\newtheorem*{SG}{The Sprague-Grundy Theorem}
\newtheorem{observation}{Observation}

\theoremstyle{definition}

\newtheorem{conjecture}{Conjecture}

\begin{document}

\begin{center}
\uppercase{\bf Advances in Finding Ideal Play on Posets Games}
\vskip 20pt
{\bf A. Clow}\\
{\smallit Department of Mathematics and Statistics, St.Francis Xavier University, Antigonish, Nova Scotia, Canada}\\
\vskip 10pt
{\bf S. Finbow\footnote{This research was funded by Natural Sciences and Engineering Research Council of Canada grant numbers 2014-06571. }}\\
{\smallit Department of Mathematics and Statistics, St.Francis Xavier University, Antigonish, Nova Scotia, Canada}\\
\end{center}
\vskip 20pt
\centerline{\smallit Received: , Revised: , Accepted: , Published: } 
\vskip 30pt

\centerline{\bf Abstract}
\noindent
Poset games are a class of combinatorial game that remain unsolved.  Soltys and Wilson  proved that computing wining strategies is in \textbf{PSPACE} and aside from special cases such as Nim and N-Free games,  \textbf{P} time algorithms for finding ideal play are unknown. This paper presents methods calculate the nimber of posets games allowing for the classification of winning or losing positions. The results present an equivalence of ideal strategies on posets that are seemingly unrelated.  

\pagestyle{myheadings} 
\markright{\smalltt INTEGERS: 21 (2021)\hfill} 
\thispagestyle{empty} 
\baselineskip=12.875pt 
\vskip 30pt

\section{Introduction}

\textit{Posets Games} are impartial combinatorial games whose game boards are partially ordered set (posets) \(P\)  on which players take turns removing an element \(p\in P\) and every element \(p'\geq_P p\) from \(P\).  We define \(P-p_\leq=P\setminus \{p':p\leq p'\}\). Each turn a player must remove an element if they can. In this paper we consider \textit{normal play} games, where the the first player to have no move loses.  Examples of poset game include Nim, Chomp, Subset Take-Away, Divisors and Geography.

\begin{figure}[h!]\label{F1}
\centering
\scalebox{0.5}{
\begin{tikzpicture}

\draw[thick,black] (-0.5,-0.5) rectangle (10.5,3.5);

    \draw(0,0)--(1,1);
    \draw(2,0)--(1,1)--(1,2)--(3,3)--(4,2)--(2,0);
    \draw(4,0)--(2,2)--(3,3);

    \filldraw[black] (0,0) circle[radius=1mm];
    \filldraw[black] (2,0) circle[radius=1mm];
    \filldraw[black] (4,0) circle[radius=1mm];
 
    \filldraw[black] (1,1) circle[radius=1mm];
    \filldraw[red] (3,1) circle[radius=1mm];
 
    \filldraw[black] (1,2) circle[radius=1mm];
    \filldraw[black] (2,2) circle[radius=1mm];
    \filldraw[black] (4,2) circle[radius=1mm];
 
    \filldraw[black] (3,3) circle[radius=1mm];
 
    \draw[->](4.5,1.5)--(6,1.5);
 
    \draw(6,0)--(7,1)--(7,2)--(7,1)--(8,0);
 
    \filldraw[black] (6,0) circle[radius=1mm];
    \filldraw[black] (8,0) circle[radius=1mm];
    \filldraw[black] (10,0) circle[radius=1mm];
 
    \filldraw[black] (7,1) circle[radius=1mm];

    \filldraw[black] (7,2) circle[radius=1mm];

\draw[thick,black] (11.5,-0.5) rectangle (22.5,3.5);

    \draw(12,2)--(13,0)--(14,2)--(15,0)--(16,2);
    \draw(21,2)--(21,0);

    \filldraw[black] (12,2) circle[radius=1mm];
    \filldraw[black] (13,0) circle[radius=1mm];
    \filldraw[black] (14,2) circle[radius=1mm];
    \filldraw[red] (15,0) circle[radius=1mm];
    \filldraw[black] (16,2) circle[radius=1mm];

    \draw[->](17,1.5)--(19.5,1.5);

    \filldraw[black] (21,2) circle[radius=1mm];
    \filldraw[black] (21,0) circle[radius=1mm];

\draw[thick,black] (-0.5,-5.5) rectangle (10.5,-1.5);

    \draw(0,-4.5)--(1,-3)--(2,-4.5);
    \draw(2.5,-4.5)--(3.5,-3)--(4.5,-4.5);
    
    \draw(7,-4.5)--(8,-3)--(9,-4.5);

    \draw[->](5,-3.5)--(6.5,-3.5);

    \filldraw[black] (0,-4.5) circle[radius=1mm];
    \filldraw[black] (1,-3) circle[radius=1mm];
    \filldraw[black] (2,-4.5) circle[radius=1mm];
    
    \filldraw[black] (2.5,-4.5) circle[radius=1mm];
    \filldraw[black] (3.5,-3) circle[radius=1mm];
    \filldraw[red] (4.5,-4.5) circle[radius=1mm];
    
    \filldraw[black] (7,-4.5) circle[radius=1mm];
    \filldraw[black] (8,-3) circle[radius=1mm];
    \filldraw[black] (9,-4.5) circle[radius=1mm];
    
    \filldraw[black] (10,-4.5) circle[radius=1mm];

\draw[thick,black] (11.5,-5.5) rectangle (22.5,-1.5);

    \draw(14,-5)--(12,-3.5)--(14,-2)--(16,-3.5)--(14,-5);
    
    \draw[->](17,-3.5)--(19.5,-3.5);

    \filldraw[red] (14,-5) circle[radius=1mm];
    \filldraw[black] (12,-3.5) circle[radius=1mm];
    \filldraw[black] (14,-2) circle[radius=1mm];
    \filldraw[black] (16,-3.5) circle[radius=1mm];
\end{tikzpicture}
}

    \caption{Examples of moves in poset games}
    \label{ex.1.1}
\end{figure}
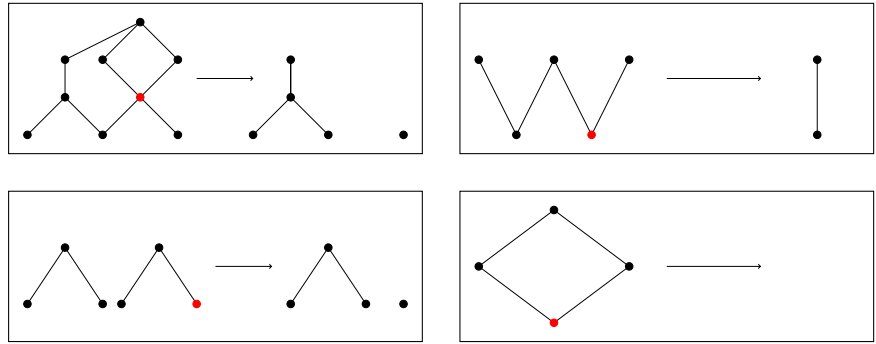

Nim, central to the theory of all impartial games, is a poset game played on any number of disjoint totally ordered sets of finite or transfinite cardinality called piles (see Figure~\ref{ex.1.2}) and as a result serves as one of the simplest poset games. It is well known that a game of Nim  is in \(\mathcal{P}\) (the set of previous player win games, also called a $\mathcal{P}$-positon) if and only if the binary XOR sum of the piles heights is 0 (zero) \cite{conway2003winning}.

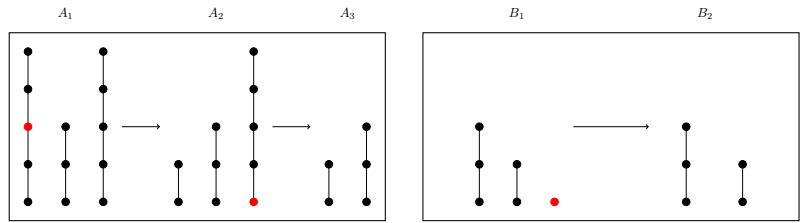
\begin{figure}[h!]\label{F2}
\centering
 \scalebox{0.5}{
        \begin{tikzpicture}

      \draw[black,thick] (0,0) rectangle (10,5);

            \draw (0.5,0.5)--(0.5,4.5);
            \draw (1.5,0.5)--(1.5,2.5);
            \draw (2.5,0.5)--(2.5,4.5);
            
            \draw (4.5,0.5)--(4.5,1.5);
            \draw (5.5,0.5)--(5.5,2.5);
            \draw (6.5,0.5)--(6.5,4.5);
            
            \draw (8.5,0.5)--(8.5,1.5);
            \draw (9.5,0.5)--(9.5,2.5);
            
            \draw[->](3,2.5)--(4,2.5);
            \draw[->](7,2.5)--(8,2.5);
            
	   \node[fill=none] at (1.5,5.5) (nodes) {\(A_1\)};
            \node[fill=none] at (5.5,5.5) (nodes) {\(A_2\)};
            \node[fill=none] at (9,5.5) (nodes) {\(A_3\)};
    
            \filldraw[black] (0.5,0.5) circle[radius=1mm];
            \filldraw[black] (0.5,1.5) circle[radius=1mm];
            \filldraw[red] (0.5,2.5) circle[radius=1mm];
            \filldraw[black] (0.5,3.5) circle[radius=1mm];
            \filldraw[black] (0.5,4.5) circle[radius=1mm];
            
            \filldraw[black] (1.5,0.5) circle[radius=1mm];
            \filldraw[black] (1.5,1.5) circle[radius=1mm];
            \filldraw[black] (1.5,2.5) circle[radius=1mm];
            
            \filldraw[black] (2.5,0.5) circle[radius=1mm];
            \filldraw[black] (2.5,1.5) circle[radius=1mm];
            \filldraw[black] (2.5,2.5) circle[radius=1mm];
            \filldraw[black] (2.5,3.5) circle[radius=1mm];
            \filldraw[black] (2.5,4.5) circle[radius=1mm];

            \filldraw[black] (4.5,0.5) circle[radius=1mm];
            \filldraw[black] (4.5,1.5) circle[radius=1mm];
            
            \filldraw[black] (5.5,0.5) circle[radius=1mm];
            \filldraw[black] (5.5,1.5) circle[radius=1mm];
            \filldraw[black] (5.5,2.5) circle[radius=1mm];
            
            \filldraw[red] (6.5,0.5) circle[radius=1mm];
            \filldraw[black] (6.5,1.5) circle[radius=1mm];
            \filldraw[black] (6.5,2.5) circle[radius=1mm];
            \filldraw[black] (6.5,3.5) circle[radius=1mm];
            \filldraw[black] (6.5,4.5) circle[radius=1mm];  
            
            \filldraw[black] (8.5,0.5) circle[radius=1mm];
            \filldraw[black] (8.5,1.5) circle[radius=1mm];
            
            \filldraw[black] (9.5,0.5) circle[radius=1mm];
            \filldraw[black] (9.5,1.5) circle[radius=1mm];
            \filldraw[black] (9.5,2.5) circle[radius=1mm];

        \draw[black,thick] (11,0) rectangle (21,5);

            \draw(12.5,0.5)--(12.5,2.5);
            \draw(13.5,0.5)--(13.5,1.5);
            
            \draw[->](15,2.5)--(17,2.5);
            
            \draw(18,0.5)--(18,2.5);
            \draw(19.5,0.5)--(19.5,1.5);

            \filldraw[black] (12.5,0.5) circle[radius=1mm];
            \filldraw[black] (12.5,1.5) circle[radius=1mm];
            \filldraw[black] (12.5,2.5) circle[radius=1mm];
            
            \filldraw[black] (13.5,0.5) circle[radius=1mm];
            \filldraw[black] (13.5,1.5) circle[radius=1mm];
            
            \filldraw[red] (14.5,0.5) circle[radius=1mm];

            \filldraw[black] (18,0.5) circle[radius=1mm];
            \filldraw[black] (18,1.5) circle[radius=1mm];
            \filldraw[black] (18,2.5) circle[radius=1mm];
            
            \filldraw[black] (19.5,0.5) circle[radius=1mm];
            \filldraw[black] (19.5,1.5) circle[radius=1mm];

	   \node[fill=none] at (13.5, 5.5) (nodes) {\(B_1\)};
            \node[fill=none] at (18.5,5.5) (nodes) {\(B_2\)};

        \end{tikzpicture}
    }
    
    \caption{Examples of moves in Nim}
    \label{ex.1.2}
\end{figure}

Aside from Nim  very little is known about ideal play on general poset games. Soltys and Wilson \cite{soltys2011complexity} proved that computing wining strategies is in \textbf{PSPACE} and aside from special cases such as Nim and N-Free games \cite{fenner2015combinatorial}, which are not the focus of this paper, \textbf{P} time algorithms for finding ideal play are unknown.  Byrnes \cite{byrnes2003poset} also proved non-constructively that local periodicity exists in Chomp and poset games that resemble Chomp. Attempts at constructive results have thus far been largely unsuccessful \cite{zeilberger2004chomp}. Computational efforts, like those of Zeilburger \cite{zeilberger2001three} demonstrates that this local periodicity leads to no discernible global pattern even in cases as small as 3 by \(n\) Chomp. 

One approach is to make use of the Sprague Grundy function which recursively assigns a non-negative integer (or \textit{nimber}) to each impartial game. Let \(G\) be the set of all impartial games and let \(\emptyset\equiv\)the empty-game.  Define the \textit{nimber} of \(A\), denoted \(\mathcal{G}(A)\), recursively in the universe of non-negative integers as follows:
\begin{itemize}
    \item Let \(A\in G\), \(\mathcal{G}(A)=Min\{\mathcal{G}(A'):\) \(A'\) one move from \(A\}^C\)
    \item \(\mathcal{G}(\emptyset)=0\).
\end{itemize}
For an impartial game \(A\), define we define the \textit{option value set of $A$} , to be the set \(A^*=\{\mathcal{G}(A'):\) \(A'\) one move from \(A\}\).  We will call  \(A\)   \textit{weakly-canonical} if  \(|A^*|=\mathcal{G}(A)\).   If there exists a set of legal moves that takes the game \(B\) to the game \(A\), we will say \(A\) is a \textit{follower} of \(B\) and denote this \(A<B\). As we are only dealing with impartial games there should be no confusion between this and the game lattice or the ordering of Partisan games. The minimum excluded value of a set $S$ is the least nonnegative integer which is not included in $S$ and is denoted $mex(S)$. Note that $\mathcal{G}(A)=mex(A^*)$. 

\begin{SG}[\cite{albert2019lessons}]
An impartial combinatorial game \(A\) is a \(\mathcal{P}\) position if and only if \(\mathcal{G}(A)=0\).
\end{SG}

An equivalent formation of the Sprague Grundy Theorem (given below) provides insights on how to approach so called disjoint games. Games are said to be disjoint or a sum of games, if a single game is created from two disjoint games such that a player may make a move on either game but not both during a turn \cite{conway2003winning}. This idea is applicable in all combinatorial games but has a particularly nice expression when limited to poset games. A poset game \(P\) that is made up of disjoint pieces \(A\) and \(B\) can be expressed simply as the poset game played on \(A+B\). 

More formally, a \textit{Fence} \(F\) is a Poset \(F=\{f_0,f_1,f_2,..,f_n\}\) such that \(f_0>f_1,f_1<f_2,f_2>f_3,f_3<f_4,...f_{n-1}<f_n\) or \(f_0<f_1,f_1>f_2,f_2<f_3,f_3>f_4,...f_{n-1}>f_n\) if \(n\) is even and \(f_0>f_1,f_1<f_2,f_2>f_3,f_3<f_4,...f_{n-1}>f_n\) or \(f_0<f_1,f_1>f_2,f_2<f_3,f_3>f_4,...f_{n-1}<f_n\) if \(n\) is odd, and such that these are all comparabilities between the points. The points \(f_0\) and \(f_n\) are the \textit{endpoints} of the fence. A Poset \(P\) is \textit{connected} if and only if for all \(p_1,p_2\in P\), there is a fence \(F\subset P\) with an endpoints \(p_1,p_2\). A Poset that is not connected is called \textit{disconnected} \cite{schroder2003ordered}.


\begin{SG}[\cite{albert2019lessons}]
Every impartial game \(A\) is equivalent in terms of being summed as a disjoint game to a pile in Nim of height \(m\), if and only if \(\mathcal{G}(A)=m\). 
\end{SG}

\begin{corollary} 
Let \(A\) and \(B\) be impartial combinatorial games. Then \(\mathcal{G}(A+B)=\mathcal{G}(A)\oplus \mathcal{G}(B)\), where \(\oplus\) denotes binary XOR.
\end{corollary}

In this paper we provide insights into how to play on connected poset games.  We do so by factoring/partitioning  a connected poset \(A\)  into subposets that give meaningful information about the nimber \(\mathcal{G}(A)\) and the subgames of \(A\). Combining results with the Sprague-Grundy Theorem to provide insights into a larger class of poset games.   The work is related to playing games on the ordinal sum of Posets, but generalizes this idea. Let $A$ and $B$ 
 be posets with disjoint underlying sets. Then the ordinal sum $A:B$
 is the poset on $A\cup B$
 with 
$x\le y$ if, either $x, y\in A$
 and $x\le_A y$
; or 
$x, y\in B$ and $x\le_B y$
; or $x\in A$
 and $y\in B$.  In other word any move in $A$, eliminates all possible moves in $B$ for the remainder for the game.  This concept has been generalized to impartial games.  The following result is due to Fisher, Nowakowski and Santos \cite{fisher2018sterling} where, \(mex(S,k)=mex(S')\), for \(S'=S\cup \{mex(S,0),mex(S,1),...,mex(S,k-1)\}\) and \(mex(S,0)=mex(S)\).
 
 \begin{theorem}[\cite{fisher2018sterling}]\label{rjn}
Let \(G,H\) be impartial games. Then, $ \mathcal{G}(G : H ) = mex( G^*,\mathcal{G}(H)).$
\end{theorem}

In Section 2, we define a mapping between poset which is preserves the underlying order of the poset.  This map is used to partition (or factor) the poset and establish a relationship between this partition and the nimbers in Section 3. Sections 4 provides examples of applications of the results and Section 5 provides some concluding remarks and directions for future research.

\section{Order Compressing Map}

%

We define a function \(f:P\rightarrow Q\)  such that \(P\) and \(Q\) are partially ordered sets, to be  \textit{order compressing} if for all \(x,y\in P\), \(f(x)=f(y)=q\in Q\) if and only if for every $z\in P$,
\begin{itemize}
    \item if \(z<_P x\) and \(z<_P y\), then \(f(z)\leq_Q q\)
    \item if \(z\nless_P x\) and \(z\nless_P y\), then \(f(z)\nless_Q q\)
    \item if 
\(  \left\{ \begin{array}{ll}
        z<_P x \mbox{ and } z\nless_P y;\\
        z\nless_P x \mbox{ and } z<_P y\end{array} \right\} \), then \(f(z)=q\)
\end{itemize}

  Figure~\ref{ex.2.1} gives an example of an order compressing function. Each order compressing function is clearly a homomorphism however every order compressing function is also order reflecting. In fact the third a condition given makes order compressions an even stronger condition than simply being an order reflecting homomorphism.  Figure~\ref{ex.2.2} gives an example of a homomorphism which is not order compressing.

\begin{figure}[h!]
    \centering

\scalebox{0.8}{ 
    \begin{tikzpicture}
    
    
    	\draw(1.5,0.5)--(0.5,1.5)--(0.5,3.5)--(1.5,5)--(3.5,3.5)--(4.5,2.5)--(3.5,1.5)--(1.5,0.5);
       	\draw (3.5,1.5)--(2.5,2.5)--(3.5,3.5)--(2.5,5)--(0.5,3.5);

            \filldraw[green] (1.5,0.5) (1.4,0.4) -- (1.6,0.4) -- (1.5,0.6) -- cycle;

            \filldraw[blue] (0.5,1.5)  (0.4,1.4) --(0.6,1.4) -- (0.6,1.6) -- (0.4,1.6) -- cycle;
            \filldraw[blue] (0.5,2.5)  (0.4,2.4) --(0.6,2.4) -- (0.6,2.6) -- (0.4,2.6) -- cycle;
            \filldraw[blue] (0.5,3.5)  (0.4,3.4) --(0.6,3.4) -- (0.6,3.6) -- (0.4,3.6) -- cycle;

            \filldraw[purple] (2.5,2.5) circle[radius=1mm];
            \filldraw[purple] (3.5,3.5) circle[radius=1mm];
            \filldraw[purple] (3.5,1.5) circle[radius=1mm];
            \filldraw[purple] (4.5,2.5) circle[radius=1mm];

            \filldraw[orange] (2.5,5) (2.45,4.9) --(2.55,4.9) -- (2.6,5.05) -- (2.5,5.15) --(2.4,5.05) -- cycle;
            \filldraw[orange] (1.5,5) (1.45,4.9) --(1.55,4.9) -- (1.6,5.05) -- (1.5,5.15) --(1.4,5.05) -- cycle;

        \draw[->](5,2.5)--(7,2.5);
        \node[fill=none] at (5.9,3) (nodes) {\(f_1\)};

	\draw(8.5,0.5)--(7.5,2.5)--(8.5,4.5)--(9.5,2.5)--(8.5,0.5);
		
            \filldraw[green] (8.5,0.5) (8.4,0.4) -- (8.6,0.4) -- (8.5,0.6) -- cycle;
            \filldraw[blue] (7.5,2.5) (7.4,2.4) --(7.6,2.4) -- (7.6,2.6) -- (7.4,2.6) -- cycle;
            \filldraw[orange] (8.5,4.5) (8.45,4.4) --(8.55,4.4) -- (8.6,4.55) -- (8.5,4.65) --(8.4,4.55) -- cycle;
            \filldraw[purple] (9.5,2.5) circle[radius=1mm];

    \end{tikzpicture}
  }

    \caption{An order compressing homomorphism}
    \label{ex.2.1}
\end{figure}
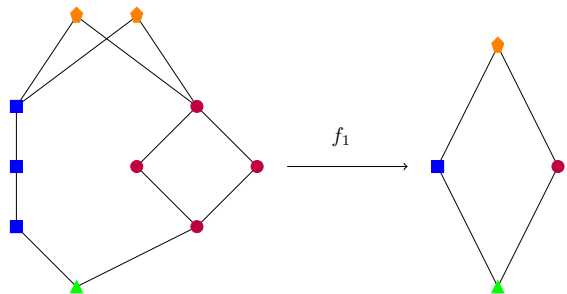

\begin{figure}[h!]
    \centering

\begin{tikzpicture}
	
		\draw (-3, 2)--(-2, 0);
		\draw (-2, 0)--(-1, 2);
		\draw (-1, 2) -- (0, 0);
		
		 \filldraw[red] (-3, 2)  circle[radius=1mm];
		 \filldraw[blue]  (-2, 0) (-2.1,-0.1) -- (-1.9,-0.1) -- (-1.9,0.1) -- (-2.1,0.1) -- cycle;
		 \filldraw[red]   (-1, 2) circle[radius=1mm];
		 \filldraw[blue]  (0, 0) (-0.1,-0.1) -- (0.1,-0.1) -- (0.1,0.1) -- (-0.1,0.1) -- cycle;

      \draw[->](1,1)--(3,1);
        \node[fill=none] at (2,1.5) (nodes) {\(f_2\)};
        
        \draw (4,2)--(4,0);
    	\filldraw[red]  (4,2) circle[radius=1mm];
        \filldraw[blue] (4,0) (3.9,-0.1) -- (4.1,-0.1) -- (4.1,0.1) -- (3.9,0.1) -- cycle;


\end{tikzpicture}

    \caption{A map which is a homomorphism, but not order compressing}
    \label{ex.2.2}
\end{figure}
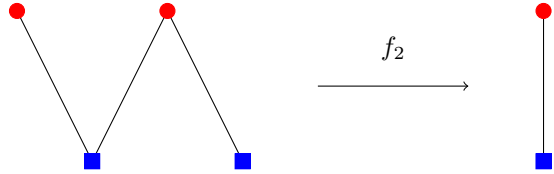

When \(f:P\rightarrow Q\) is order compressing, an {\it $f$-factor} of \(P\) is a subposet of \(P\) defined by \(f^{-1}(x)\) for some \(x\in Q\). The set of $f$-factors is a {\it $Q$-factorization} of $P$.  When the choice of $f$ is clear from the context we call an $f$-factor of $P$ a {\it factor of $P$} and $Q$-factorization of $P$ a {\it factorization} of $P$. In Figure~\ref{ex.2.1} each factor is given a unique colour and shape, as is the convention for the rest of the paper.

\section{Equivalencies of Games}

In this section we establish that order compressing maps which preserve nimbers in individuals factors, preserves the nimber of the entire poset.  We start with the following observation.

\begin{lemma}\label{not weaklycanonical}
Let \(P\) be a poset, if for all \(A<P\), \(\mathcal{G}(A)\neq \mathcal{G}(P)\), then \(P\) is weakly-canonical.
\end{lemma}

\begin{proof}
Suppose $P$ is not weakly-canonical. Then there exists a $p\in P$ so that \(\mathcal{G}(P-(p)_\leq)>\mathcal{G}(P)\).  Hence there is an option in $P-(p)_\leq$ to a game $A$ with $\mathcal{G}(A)= \mathcal{G}(P)$.
%
\end{proof}

\begin{observation}\label{restriction}
Let \(f:P\rightarrow Q\) be an order compressing function. Then \(f|_{P-p_\leq}\) is an order compressing function.
\end{observation}

\begin{proof}
Let \(x,y\in P-p_\leq\).  Then since $f$ is order compressing, $f(x)=f(y)$ if and only for each $z\in P$, one of the three conditions in the definition of an order compressing function hold.  It follows that  $f(x)=f(y)$ if and only for each $z\in P-p_\leq\subset P$, one of the three conditions in the definition of an order compressing function hold. 
%
\end{proof}

We are ready to give our first main result.

\begin{theorem}\label{thm 1}
Let \(f:A\rightarrow Q\) and \(g:B\rightarrow Q\) be order compressing maps such that \(f(x)=\alpha\) where \(x\) is maximal in \(A\) and for all $\beta\in Q$ with \(\beta\neq \alpha\), \(f^{-1}(\beta)=g^{-1}(\beta)\). If $\mathcal{G}(f^{-1}(\alpha))=\mathcal{G}(g^{-1}(\alpha))$, then
\[
\mathcal{G}(A)=\mathcal{G}(B).
\]
\end{theorem}

\begin{proof}
%
%
%

Let $Q$ and $\alpha$ be fixed.  Note that if $A=\emptyset$ or $B=\emptyset$ then the result is trivially true. Let $A$ be a minimal counterexample.  That is for all \(A'<A\), \(A'\) is not a counterexample. Let \(S\) be the set of all posets \(P\) where there exists an order compression \(h:P\rightarrow Q\) such that for all \(\beta\neq \alpha\in Q\), \(f^{-1}(\beta)=h^{-1}(\beta)\) and \(\mathcal{G}(f^{-1}(\alpha))=\mathcal{G}(h^{-1}(\alpha))\). Assume $B$ is a poset in $S$ so that for all \(B'<B\) such that \(B'\in S\),  \(\mathcal{G}(A)=\mathcal{G}(B')\).


By assumption, \(A\setminus  f^{-1}(\alpha)=B\setminus g^{-1}(\alpha)\). Let $a$ be an element of the poset $A$ so that $a\notin f^{-1}(\alpha)$.  Then $a$ is an element of $ B\setminus g^{-1}(\alpha)$.  The restriction of  $f$ and $g$ to $A-a_\leq$ and $B-a_\leq$ is an order compressing function, by Observation~\ref{restriction}.  Further the $Q$-factorization on these posets induced by the restrictions are identical for each element of $Q$, except possibly $\alpha$. Hence, by the minimality of our counter example, it follows that \(\mathcal{G}(A-a_\leq)=\mathcal{G}(B-a_\leq)\). It follows that for some $a\in  f^{-1}(\alpha)$, $\mathcal{G}(A-a_\leq)=\mathcal{G}(B)$ or for some $b\in B\setminus g^{-1}(\alpha)\), $(\mathcal{G}(B-b_\leq)=\mathcal{G}(A)$. 


\medskip
\noindent {\bf Claim 1}: For all \(c\in f^{-1}(\alpha)\) and \(d\in g^{-1}(\alpha)\) such that \(\mathcal{G}(f^{-1}(\alpha)-c_\leq)=\mathcal{G}(g^{-1}(\alpha)-d_\leq)\), \(\mathcal{G}(A-c_\leq)=\mathcal{G}(B-d_\leq)\).

\smallskip 
\begin{proof} The claim follows trivially from the minimality of $A$ and $B$.\end{proof}

\bigskip
Suppose both \(f^{-1}(\alpha)\) and \(g^{-1}(\alpha)\) are weakly-canonical and recall that $\mathcal{G}(f^{-1}(\alpha))=\mathcal{G}(g^{-1}(\alpha))$.  It follows that for all \(c\in f^{-1}(\alpha)\), there exists a \(d\in g^{-1}(\alpha)\) and for all \(d\in g^{-1}(\alpha)\), there exists a \(c\in f^{-1}(\alpha)\) such that \(\mathcal{G}(f^{-1}(\alpha)-c_\leq)=\mathcal{G}(g^{-1}(\alpha)-d_\leq)\). From Claim 1,  $A^*=B^*$ and hence \(\mathcal{G}(A)=\mathcal{G}(B)\). Therefore we consider the following two cases.

\medskip
\noindent Case 1: \(f^{-1}(\alpha)\) is not weakly-canonical.

\medskip
By Lemma~\ref{not weaklycanonical} there exists a game $Y<f^{-1}(\alpha)$ such that $\mathcal{G}(Y)=\mathcal{G}(f^{-1}(\alpha))$. Hence there is a set of moves on $f^{-1}(\alpha)$ which result in the game $Y$.  Playing this identical set of moves on the poset $A$, we obtain a poset $A'<A$.  Let $f_R$ be the restriction of $f$ to $A'$.  Then as $\alpha$ is maximal, \(f_R^{-1}(\beta)=f^{-1}(\beta)\) for all \(\beta\neq \alpha\) and \(f_R^{-1}(\alpha)=Y\), so $\mathcal{G}(f_R^{-1}(\alpha))= \mathcal{G}(f^{-1}(\alpha))$. Hence, \(A'\in S\) and \(A'<A\) so \(A'\) is not a counterexample. But since \(A'\) is not a counterexample, \(\mathcal{G}(A')=\mathcal{G}(A)$ and $\mathcal{G}(A')=\mathcal{G}(B)\). Hence $\mathcal{G}(A)=\mathcal{G}(B)$.

\medskip
\noindent Case 2: \(g^{-1}(\alpha)\) is not weakly-canonical and \(f^{-1}(\alpha)\) is weakly-canonical.

\medskip
As $\mathcal{G}(f^{-1}(\alpha))=\mathcal{G}(g^{-1}(\alpha))$, Claim 1 implies that \(A^*\subset B^*\). Since \(\mathcal{G}(A)\neq \mathcal{G}(B)\), it follows that \(\mathcal{G}(A)< \mathcal{G}(B)\) and hence there exists a \(b\in g^{-1}(\alpha)\) such that \(\mathcal{G}(g^{-1}(\alpha)-b_\leq)>\mathcal{G}(g^{-1}(\alpha))\) and \(\mathcal{G}(B-b_\leq)=\mathcal{G}(A)\). From \(\mathcal{G}(g^{-1}(\alpha)-b_\leq)>\mathcal{G}(g^{-1}(\alpha))\), it follows there exists a \(d\in g^{-1}(\alpha)-b_\leq\) such that \(\mathcal{G}(g^{-1}(\alpha)-b_\leq-d_\leq)=\mathcal{G}(g^{-1}(\alpha))\). 

By Observation~\ref{restriction}, the restriction of $g$ to $B-b_\leq-d_\leq$ is an order compressing function.  The $Q$-factorization on these posets induced by the restrictions are identical for each element of $Q$, except possibly $\alpha$. By the minimality of $B$,  
\(\mathcal{G}(B-b_\leq-d_\leq)=\mathcal{G}(A)\). But this implies that \(\mathcal{G}(B-b_\leq)\neq \mathcal{G}(A)\) which is a contradiction. Hence Case 2 is established and the Theorem follows.
\end{proof}

We now establish the converse statement is also true.

\begin{theorem}
Let \(f:A\rightarrow Q\) and \(g:B\rightarrow Q\) be order compressing maps such that \(f(x)=\alpha\) where \(x\) is maximal in \(A\) and for all $\beta\in Q$ with \(\beta\neq \alpha\in Q\), \(f^{-1}(\beta)=g^{-1}(\beta)\). If $\mathcal{G}(A)=\mathcal{G}(B)$, then
$$\mathcal{G}(f^{-1}(\alpha))=\mathcal{G}(g^{-1}(\alpha)).$$ 
\end{theorem}

\begin{proof}
Suppose that \(\mathcal{G}(f^{-1}(\alpha))\neq \mathcal{G}(g^{-1}(\alpha))\). Without loss of generality, \(\mathcal{G}(f^{-1}(\alpha))> \mathcal{G}(g^{-1}(\alpha))\). Hence there exists an \(a\in f^{-1}(\alpha)\) such that \(\mathcal{G}(f^{-1}(\alpha)-a_\leq)=\mathcal{G}(g^{-1}(\alpha))\). But from Theorem~\ref{thm 1}, this implies that \(\mathcal{G}(A-a_\leq)=\mathcal{G}(B)\), so by the definition of \(\mathcal{G}\), \(\mathcal{G}(A)\neq\mathcal{G}(B)\). 
\end{proof}

\begin{corollary}\label{cor1}
Let \(f:A\rightarrow Q\) and \(g:B\rightarrow Q\) be order compressing maps such that \(f(x)=\alpha\) where \(x\) is maximal in \(A\) and for all $\beta\in Q$ with \(\beta\neq \alpha\in Q\), \(f^{-1}(\beta)=g^{-1}(\beta)\). Then
\[
\mathcal{G}(f^{-1}(\alpha))=\mathcal{G}(g^{-1}(\alpha))\mbox{ if and only if } \mathcal{G}(A)=\mathcal{G}(B).
\]
\end{corollary}

\begin{figure}[h!]
\centering
\scalebox{1.0}{ 
	\begin{tikzpicture}

    	\draw(0.5,0.5)--(1,1.5)--(1.5,0.5);
    	\draw(4.5,1.5)--(4.5,0.5);

    	\draw[->] (2.25,1)--(3.75,1);
		\draw[->] (6.75,1)--(5.25,1);
		\node[fill=none] at (3,1.25) (nodes) {\(f\)};
		\node[fill=none] at (6,1.25) (nodes) {\(g\)};
			
		\node[fill=none] at (1,2.25) (nodes) {\(A\)};
		\node[fill=none] at (4.5,2.25) (nodes) {\(Q\)};
		\node[fill=none] at (7,2.25) (nodes) {\(B\)};
    		
    	\filldraw[blue] (0.5,0.5) (0.4,0.4) -- (0.6,0.4) -- (0.6,0.6) -- (0.4,0.6) -- cycle;
    	\filldraw[red] (1,1.5) circle [radius=1mm];
    	\filldraw[blue] (1.5,0.5) (1.4,0.4) -- (1.6,0.4) -- (1.6,0.6) -- (1.4,0.6) -- cycle;
    		
    	\filldraw[red] (4.5,1.5) circle [radius=1mm];
    	\filldraw[blue] (4.5,0.5) (4.4,0.4) -- (4.6,0.4) -- (4.6,0.6) -- (4.4,0.6) -- cycle;
    		    		
    	\filldraw[red] (7,1) circle [radius=1mm];
  	\end{tikzpicture}
}
    \caption{An example}
    \label{F5}
\end{figure}
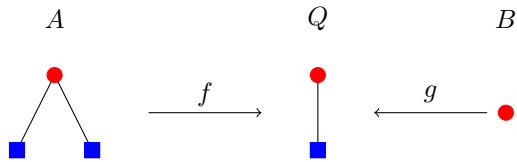

Figure~\ref{F5} gives an example of a poset where the assumptions of Theorem~\ref{thm 1} are satisfied except for the maximality of $x$ in $A$. In this example, \(\mathcal{G}(A)=2$ and $\mathcal{G}(B)=1$ and hence the assumption that \(x\) is maximal may not be relaxed in Theorem~\ref{thm 1}. Figure~\ref{F6}, gives an example of a non-trival poset.  The following Corollary implies this poset is a previous player win position.

\begin{figure}[h!]
\centering 
\scalebox{0.8}{ 
    \begin{tikzpicture}
    

	\draw(7,0)--(6,1)--(3.5,3)--(3.5,4)--(3,7);
	\draw(7,0)--(7,1)--(3.5,3)--(3.5,4)--(4,7);
	\draw(6,0)--(6,1)--(5.5,3)--(5.5,4)--(4,7);
	\draw(6,0)--(7,1)--(5.5,3)--(5.5,4)--(3,7);

	\draw(0,0)--(1,1)--(2,0)--(3,1)--(3.5,3);
	\draw(3,1)--(5.5,3);
	\draw(3,1)--(2,3);
	\draw(3,1)--(0,3);
	\draw(3,1)--(5.5,3);
	\draw(1,1)--(2,3);
	\draw(1,1)--(0,3);
	\draw(1,1)--(3.5,3);
	\draw(1,1)--(5.5,3);
	\draw(6,0)--(7,1)--(5.5,3)--(5.5,4)--(3,7);

	\draw(0,3)--(0,4)--(1,5)--(2,4)--(2,3)--(2,4)--(2.5,5);
	\draw(1,5)--(3,7);
	\draw(1,5)--(4,7);
	\draw(2.5,5)--(3,7);
	\draw(2.5,5)--(4,7);

    	\filldraw[purple] (0,0) (0.1,-0.1) -- (0,0.1) -- (-0.1,-0.1) -- cycle;
    	\filldraw[purple] (1,1) (0.9,0.9) -- (1,1.1) -- (1.1,0.9) -- cycle;
    	\filldraw[purple] (2,0) (1.9,-0.1) -- (2,0.1) -- (2.1,-0.1) -- cycle;
    	\filldraw[purple] (3,1) (3.1,0.9) -- (3,1.1) -- (2.9,0.9) -- cycle;

    	\filldraw[green] (3.5,3) (3.4,2.9) -- (3.6,2.9) -- (3.6,3.1) -- (3.4,3.1) -- cycle;
    	\filldraw[green] (3.5,4) (3.4,3.9) -- (3.6,3.9) -- (3.6,4.1) -- (3.4,4.1) -- cycle;
    	\filldraw[green] (5.5,3) (5.4,2.9) -- (5.6,2.9) -- (5.6,3.1) -- (5.4,3.1) -- cycle;
    	\filldraw[green] (5.5,4) (5.4,3.9) -- (5.6,3.9) -- (5.6,4.1) -- (5.4,4.1) -- cycle;

    	\filldraw[blue] (6,0) (5.95,-0.1) -- (6.05,-0.1) -- (6.1,0.05) -- (6,0.15) -- (5.9,0.05) -- cycle;
    	\filldraw[blue] (6,1) (5.95,0.9) -- (6.05,0.9) -- (6.1,1.05) -- (6,1.15) -- (5.9,1.05) -- cycle;
    	\filldraw[blue] (7,0) (6.95,-0.1) -- (7.05,-0.1) -- (7.1,0.05) -- (7,0.15) -- (6.9,0.05) -- cycle;
    	\filldraw[blue] (7,1) (6.95,0.9) -- (7.05,0.9) -- (7.1,1.05) -- (7,1.15) -- (6.9,1.05) -- cycle;

    	\filldraw[orange] (0,3) circle [radius=1mm];
    	\filldraw[orange] (0,4) circle [radius=1mm];
    	\filldraw[orange] (2,3) circle [radius=1mm];
    	\filldraw[orange] (2,4) circle [radius=1mm];
    	\filldraw[orange] (1,5) circle [radius=1mm];
    	\filldraw[orange] (2.5,5) circle [radius=1mm];

    	\filldraw[pink] (3,7) (3,6.9) -- (3.1,7) -- (3,7.1) -- (2.9,7) -- cycle;
    	\filldraw[pink] (4,7) (4,6.9) -- (4.1,7) -- (4,7.1) -- (3.9,7) -- cycle;

    \end{tikzpicture}
  }
    \caption{A poset and factorization whose factors all have nimber 0}
    \label{F6}
\end{figure}
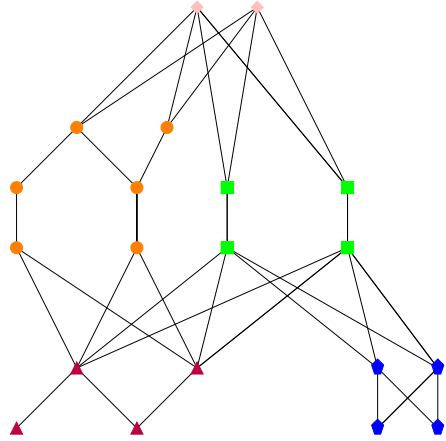

\begin{corollary}
Let \(f:A\rightarrow Q\) be an order compressing map, if for all \(x\in Q\) \(\mathcal{G}(f^{-1}(x))=0\), then \(\mathcal{G}(A)=0\)
\end{corollary}

\begin{proof}
 Let \(A_1\subset A\) such that \(A_1=A\setminus f^{-1}(f(\alpha))\) where \(\alpha\) is maximal in \(A\). By definition \(\mathcal{G}(\emptyset)=0\) so by Theorem~\ref{thm 1}, \(\mathcal{G}(A_1)=\mathcal{G}(A)\). We may now repeat this process on any maximal element of $A_1$.  It follows by induction that \(\mathcal{G}(A)=0\).
\end{proof}

To extend the results we now consider $Q$-factorizations of posets where each factor has identical option value sets.  We start with a Lemma.

\begin{lemma}\label{empty}
For all posets \(P\), \(P^*=\{\}\), if and only if \(P=\emptyset\)
\end{lemma}

\begin{proof}
Let \(P=\emptyset\), by definition \(P\) has no elements so \(P^*=\{\}\). Now let \(P\) be a Poset such that \(P^*=\{\}\). Since \(P^*\) has no elements this implies that \(P\) has no moves. But only \(\emptyset\) has no moves, so \(P=\emptyset\).
\end{proof}

\begin{theorem}\label{thm 3}
Let \(f:A\rightarrow Q\) and \(g:B\rightarrow Q\) be order compressing maps so that  for all \(\beta\in Q\), \((f^{-1}(\beta))^*=(g^{-1}(\beta))^*\). Then for any fixed but arbitrary \(\alpha\in Q\) if \(a\in f^{-1}(\alpha)\) and \(b\in g^{-1}(\alpha)\) are  such that \(\mathcal{G}(f^{-1}(\alpha)-a_\leq)=\mathcal{G}(g^{-1}(\alpha)-b_\leq)\) then \(\mathcal{G}(A-a_\leq)=\mathcal{G}(B-b_\leq)\).
\end{theorem}

\begin{proof}
Let \(A=\emptyset\), then for all \(\beta\in Q\), \(f^{-1}(\beta)=\emptyset\). By Lemma~\ref{empty}, if \((f^{-1}(\alpha))^*=(g^{-1}(\alpha))^*=\{\}\), then \(f^{-1}(\alpha)=\emptyset=g^{-1}(\alpha)\). This implies that \(A=\emptyset=B\).

Let $Q$ be given and choose $A$ so that for all \(A'<A\), \(A'\) is not a counterexample. 
Fix $\alpha\in Q$ and let \(a\in f^{-1}(\alpha)\) and \(b\in g^{-1}(\alpha)\) so that \(\mathcal{G}(f^{-1}(\alpha)-a_\leq)=\mathcal{G}(g^{-1}(\alpha)-b_\leq)\).  By Observation~\ref{restriction}, the restriction of $f$ to $A-a_\leq$, say $f_R$, and the restriction of $g$ to $B-b_\leq$, say $g_R$ are order compressing functions.


By assumption \(A-a_\leq\) and \(B-b_\leq\) have the same option value sets for all \(f_R\) and \(g_R\)-factors except possibly at the now maximal \(f_R\) and \(g_R\)-factor \(f_R^{-1}(\alpha)\) and \(g_R^{-1}(\alpha)\). Define the poset $C$ and an order compressing map \(h:C\rightarrow Q\) so that $h^{-1}(\alpha)=f_R^{-1}(\alpha)$ and for all $\beta\neq \alpha$, $h^{-1}(\beta)=g_R^{-1}(\beta)$. On one hand, by Theorem~\ref{thm 1},   $\mathcal{G}(B-b_\leq)=\mathcal{G}(C).$ On the other hand, $h^{-1}(\alpha)=f_R^{-1}(\alpha)$ and for all $\beta\neq \alpha$, $(h^{-1}(\beta))^*=(g_R^{-1}(\beta))^*$.  As $A-a_\leq<A$, $A-a_\leq$ is not a counterexample and it follows that $(A-a_\leq)^*=C^*$ and hence \(\mathcal{G}(A-a_\leq)=\mathcal{G}(C)\).  Therefore $\mathcal{G}(A-a_\leq)=\mathcal{G}(B-b_\leq)$.

\end{proof}

\begin{corollary}\label{cor 4}
Let \(f:A\rightarrow Q\) and \(g:B\rightarrow Q\) be order compressing maps so that  for all \(\beta\in Q\), \((f^{-1}(\beta))^*=(g^{-1}(\beta))^*\). Then \(A^*=B^*\).
\end{corollary}

\begin{proof}
Let \(f:A\rightarrow Q\) and \(g:B\rightarrow Q\) be order compressing maps so that  for all \(\beta\in Q\), \((f^{-1}(\beta))^*=(g^{-1}(\beta))^*\). \(A\) and \(B\) satisfy the assumptions of Theorem~\ref{thm 3} so for all elements \(a\in A\) the exists a  \(b\in B\) such that \(\mathcal{G}(A-a_\leq)=\mathcal{G}(B-b_\leq)\). By the same argument for all \(b\in B\) there exists an \(a\in A\) such that \(\mathcal{G}(A-a_\leq)=\mathcal{G}(B-b_\leq)\). This completes the proof.
\end{proof}

\section{Applications}

In this section we provide two examples of applications of the results of the previous section.  Consider the posets drawn in Figure~\ref{F100}.  We claim these all have the same nimber.

\begin{figure}[h!]
\centering
  \scalebox{0.6}{ 
    \begin{tikzpicture}

\draw[thick,black] (-0.5,-0.5) rectangle (4.5,5.5);

	\draw(0,2)--(0,3)--(0.5,4)--(1,3)--(1,2)--(3.5,3);
	\draw(4,4)--(3.5,3)--(2.75,4)--(2,3)--(1,2);
	\draw(0,1)--(0,2);
	\draw(0,1)--(1,2);
	\draw(2,1)--(0,2);
	\draw(2,1)--(1,2);
	\draw(4,1)--(0,2);
	\draw(4,1)--(1,2);

    	\filldraw[blue] (0,2) circle [radius=1mm];
    	\filldraw[blue] (0,3) circle [radius=1mm];
    	\filldraw[blue] (0.5,4) circle [radius=1mm];
    	\filldraw[blue] (1,2) circle [radius=1mm];
    	\filldraw[blue] (1,3) circle [radius=1mm];
    	\filldraw[blue] (2,3) circle [radius=1mm];
    	\filldraw[blue] (2.75,4) circle [radius=1mm];
    	\filldraw[blue] (3.5,3) circle [radius=1mm];
    	\filldraw[blue] (4,4) circle [radius=1mm];
    	
	\draw(0,1)--(1,0)--(2,1)--(3,0)--(4,1);

    	\filldraw[red] (1,0) (0.9,-0.1) -- (1.1,-0.1) -- (1.1,0.1) -- (0.9,0.1) -- cycle;
    	\filldraw[red] (3,0) (2.9,-0.1) -- (3.1,-0.1) -- (3.1,0.1) -- (2.9,0.1) -- cycle;
    	\filldraw[red] (0,1) (-0.1,0.9) -- (0.1,0.9) -- (0.1,1.1) -- (-0.1,1.1) -- cycle;
    	\filldraw[red] (2,1) (1.9,0.9) -- (2.1,0.9) -- (2.1,1.1) -- (1.9,1.1) -- cycle;
    	\filldraw[red] (4,1) (3.9,0.9) -- (4.1,0.9) -- (4.1,1.1) -- (3.9,1.1) -- cycle;

\draw[thick,black] (5.5,-0.5) rectangle (10.5,5.5);

	\draw(7,1)--(8,0)--(9,1)--(6,2)--(6,3)--(6.5,4)--(7,3)--(7,2)--(8,3)--(8.75,4)--(9.5,3)--(10,4);
	\draw(9,1)--(7,2);
	\draw(7,1)--(7,2);
	\draw(7,1)--(6,2);
	\draw(9.5,3)--(7,2);

    	\filldraw[blue] (6,2) circle [radius=1mm];
    	\filldraw[blue] (6,3) circle [radius=1mm];
    	\filldraw[blue] (6.5,4) circle [radius=1mm];
    	\filldraw[blue] (7,2) circle [radius=1mm];
    	\filldraw[blue] (7,3) circle [radius=1mm];
    	\filldraw[blue] (8,3)  circle [radius=1mm];
    	\filldraw[blue] (8.75,4) circle [radius=1mm];
    	\filldraw[blue] (9.5,3) circle [radius=1mm];
    	\filldraw[blue] (10,4) circle [radius=1mm];

    	\filldraw[red] (8,0) (7.9,-0.1) -- (8.1,-0.1) -- (8.1,0.1) -- (7.9,0.1) -- cycle;
    	\filldraw[red] (7,1) (6.9,0.9) -- (7.1,0.9) -- (7.1,1.1) -- (6.9,1.1) -- cycle;
    	\filldraw[red] (9,1) (8.9,0.9) -- (9.1,0.9) -- (9.1,1.1) -- (8.9,1.1) -- cycle;

\draw[thick,black] (11.5,-0.5) rectangle (14.5,5.5);
    	
	\draw(13,2)--(12,1)--(13,0)--(14,1)--(13,2);

		\filldraw[blue] (13,2) circle [radius=1mm];
    	
    	\filldraw[red] (13,0) (12.9,-0.1) -- (13.1,-0.1) -- (13.1,0.1) -- (12.9,0.1) -- cycle;
    	\filldraw[red] (12,1) (11.9,0.9) -- (12.1,0.9) -- (12.1,1.1) -- (11.9,1.1) -- cycle;
    	\filldraw[red] (14,1) (13.9,0.9) -- (14.1,0.9) -- (14.1,1.1) -- (13.9,1.1) -- cycle;

\node[fill=none] at (5,2.5) (nodes) {$\equiv$};
\node[fill=none] at (11,2.5) (nodes) {$\equiv$};
    \end{tikzpicture}
  }
  \caption{An example of three posets with nimber 3.}
  \label{F100}
\end{figure}
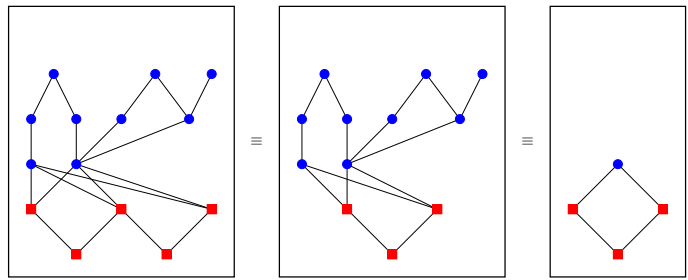

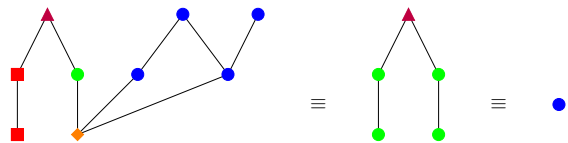
\begin{figure}[h!]
\centering
  \scalebox{0.8}{ 
    \begin{tikzpicture}

	\draw(0,2)--(0,3)--(0.5,4)--(1,3)--(1,2)--(3.5,3);
	\draw(4,4)--(3.5,3)--(2.75,4)--(2,3)--(1,2);

    	\filldraw[red] (0,2) (-0.1,1.9) -- (0.1,1.9) -- (0.1,2.1) -- (-0.1,2.1) -- cycle;
    	\filldraw[red] (0,3) (-0.1,2.9) -- (0.1,2.9) -- (0.1,3.1) -- (-0.1,3.1) -- cycle;
    	\filldraw[purple] (0.5,4) (0.6,3.9) -- (0.5,4.1) -- (0.4,3.9) -- cycle;
    	\filldraw[orange] (1,2)  (1,1.9) -- (1.1,2) -- (1,2.1) -- (0.9,2) -- cycle;
    	\filldraw[green] (1,3) circle [radius=1mm];
    	\filldraw[blue] (2,3) circle [radius=1mm];
    	\filldraw[blue] (2.75,4) circle [radius=1mm];
    	\filldraw[blue] (3.5,3) circle [radius=1mm];
    	\filldraw[blue] (4,4) circle [radius=1mm];
    	

%

    		\draw(6,2)--(6,3)--(6.5,4)--(7,3)--(7,2);

    	\filldraw[green] (6,2)  circle [radius=1mm];
    	\filldraw[green] (6,3)  circle [radius=1mm];
    	\filldraw[purple] (6.5,4) (6.6,3.9) -- (6.5,4.1) -- (6.4,3.9) -- cycle;
    	\filldraw[green] (7,2)   circle [radius=1mm];
    	\filldraw[green] (7,3) circle [radius=1mm];
%

%

    	\filldraw[blue] (9, 2.5) circle [radius=1mm];

\node[fill=none] at (5,2.5) (nodes) {$\equiv$};
\node[fill=none] at (8,2.5) (nodes) {$\equiv$};
    \end{tikzpicture}
  }
  \caption{Equivalences of the subposet of blue circles from Figure~\ref{F100} }
  \label{F101}
\end{figure}

To establish the first equivalence in Figure~\ref{F100}, consider the sub-posets  consisting of red squares.  In both cases, the the option value set of both of these sub posets in $\{0,2\}$ and hence the equivalence follows from Theorem~\ref{thm 3}.  The second equivalence in Figure~\ref{F100} follows from applying Theorem~\ref{thm 1} as we claim the sub-posets  consisting of blue circles, redrawn in Figure~\ref{F101} has nimber 1.  To see this observe that the subset of blue circles in first poset of Figure~\ref{F101} has nimber 0 and contains a maximal element of the poset. Thereom~\ref{thm 1} now implies the first equivalence in Figure~\ref{F101}.  The second equivalence in Figure~\ref{F101} can be found by applying Theorem~\ref{rjn}.

Theorem~\ref{thm 3} points to the importance of the collection of posets with the same option value set. A natural avenue of investigation is given a set of non-negative integers \(S\) is to enumerate the number of games that have option value set $S$.

\begin{theorem}\label{5}
For all \(S\subset \mathbb{N}_0\) such that \(S\neq \{\}\), \(\{P: P^*=S\}\neq \emptyset\) if and only if  \(|\{P: P^*=S\}|\) is non-finite.

\end{theorem}
\begin{proof}
If \(\{P: P^*=S\} = \emptyset\), then  \(|\{P: P^*=S\}|=0\) trivially. Assume \(Q\in \{P\in\textbf{P}: P^*=S\}\). Note the identity map on $Q$ is an order compressing function. For a given $n$, form $Q_n$ by replacing a single element of $Q$ with $2n+1$ copies of itself, so that the copies form an antichain.  It follows from Corollary~\ref{cor 4}, that $Q^*=Q_n^*$. As $n$ can be any integer, this concludes the proof.
\end{proof}

\section{Conclusion}

This paper establishes equivalences in ideal play between posets games that have obvious and not so obvious similarities.  In particular this work contributes to looking at ideal play on connected poset games.  Whether the converse of Theorem~\ref{thm 3} is true remains an open question. We end asking a related question.

Let the \textit{Lexicographic Product} of two posets, \(A\) and \(B\), denoted here $A\otimes B$, to  be the poset given by the Cartesian Product \(A\times B\) ordered by the following rule, \((a,b)\leq (a',b')\) if and only if,
\begin{itemize}
	\item \(a\leq_A a'\)
	\item if \(a=a'\), then \(b\leq_B b'\)
\end{itemize}
We note that there is a natural order compressing function $f:A\otimes B\rightarrow A$ where each $f$-factor of $A\otimes B$ is isomorphic to $B$. 

\begin{conjecture}
Let \(A,B\) be Posets Games, such that \(B=2^n\) where \(n\in\mathbb{N}_0\) and \(B\) is weakly-canonical. Then, for all \((a,b)\in A\otimes B\),
\[
\mathcal{G}(A\otimes B-(a,b)_\leq)=2^n\mathcal{G}(A-a_\leq)+\mathcal{G}(B-b_\leq)
\]
where multiplication and addition are standard for integers.
\end{conjecture}

If true this would imply for any Lexicographic Product that satisfied the assumptions given \(A\otimes B\), then \(\mathcal{G}(A\otimes B)=\mathcal{G}(A)\mathcal{G}(B)\). One of the nicest properties of which being that if the left factor \(A\) has Nimber \(0\), then \(\mathcal{G}(A\otimes B)=0\). (If the right fact \(B\) has Nimber \(0\), then as $B$ is weakly-canonical, \(A\otimes B=\emptyset\). ) Furthermore, if true then this points to the possibility of equations like that of the ordinal sum existing more broadly.

\vspace{20pt}

\noindent {\bf Acknowledgements.}

\noindent The authors would like to thank Dr. Darien DeWolfe and Dr. Richard Nowakowski for sharing their insights and thoughts on various aspects of the paper. 

%

\begin{thebibliography}{9}
\bibitem{albert2019lessons} M. H. Albert, R. J. Nowakowski, D. Wolfe, \emph{Lessons in play: an introduction to combinatorial game theory}, CRC Press, 2019.  

\bibitem{byrnes2003poset} S. Byrnes,  Poset game periodicity,  {\em Integers}, {\bf 3} (2003).

\bibitem{conway2003winning} J. H. Conway, R. K. Guy, E. R. Berlekamp, \emph{Winning ways for your mathematical plays}, AK Peters, 1983.  

\bibitem{fenner2015combinatorial} S. A. Fenner, J. Rogers, Combinatorial game complexity: an introduction with poset games,  {\em arXiv preprint arXiv:1505.07416}, (2015).


\bibitem{fisher2018sterling} M. Fisher, R. J. Nowakowski, C. Santos,  Sterling stirling play,  {Internat. J. Game Theory},  {\bf 47}(2), (2018), 557-576.

\bibitem{schroder2003ordered} B.S.W. Schr{\"o}der,   \emph{Ordered sets}, Springer, 2003.  

\bibitem{soltys2011complexity} M. Soltys, C. Wilson, On the complexity of computing winning strategies for finite poset games,  {\em Theory Comput. Syst}, {\bf 48}(3), (2011), 680-692.

\bibitem{zeilberger2004chomp} D. Zeilberger,  Chomp, recurrences and chaos,  {\em J. Difference Equ. Appl}, {\bf 10}(13-15), (2004), 1281-1293.

\bibitem{zeilberger2001three} D. Zeilberger,  Three-rowed chomp,  {\em Adv. Appl. Math}, {\bf 26}(2), (2001), 168-179.





\end{thebibliography}

\end{document}